\documentclass[12pt,reqno]{amsart}

\usepackage{amsmath}
\usepackage{amscd}
\usepackage{amssymb}
\usepackage{amstext}
\usepackage{amsthm}
\usepackage{mathrsfs}
\usepackage{a4wide}
\usepackage{latexsym}
\numberwithin{equation}{section}

\DeclareMathOperator{\dist}{dist}
\DeclareMathOperator{\supp}{supp}
\DeclareMathOperator*{\Ent}{Ent}

\newtheorem{Thm}{Theorem}[section]
\newtheorem{theorem}[Thm]{Theorem}
\newtheorem{lemma}[Thm]{Lemma}

\newtheorem{corollary}[Thm]{Corollary}

\begin{document}
\sloppy

\title[Univalent functions in model spaces: revisited]{Univalent functions in model spaces: revisited}

\author[Baranov, Belov, Borichev, Fedorovskiy]{Anton Baranov, Yurii Belov,\\Alexander Borichev, Konstantin Fedorovskiy}

\address{Anton Baranov:
\newline Department of Mathematics and Mechanics, St.~Petersburg State University, St.~Petersburg, Russia;
\newline National Research University Higher School of Economics, St.~Petersburg, Russia
\newline {\tt anton.d.baranov@gmail.com}
\smallskip
\newline \indent Yurii Belov:
\newline St.~Petersburg State University, St.~Petersburg, Russia;
\newline Bauman Moscow State Technical University, Russia
\newline {\tt j\_b\_juri\_belov@mail.ru}
\smallskip
\newline \indent Alexander Borichev:
\newline Aix-Marseille Universit\'e, CNRS, Centrale Marseille, I2M, Marseille, France
\newline {\tt alexander.borichev@math.cnrs.fr}
\smallskip
\newline \indent Konstantin Fedorovskiy:
\newline Bauman Moscow State Technical University, Moscow, Russia;
\newline St.~Petersburg State University, St.~Petersburg, Russia
\newline {\tt kfedorovs@yandex.ru}
}
%\thanks{The work was supported by
%the joint project 17-51-150005-NCNI-a of the Russian Foundation for Basic
%Research and CNRS (France). The first and the third authors were also
%supported by the project 16-01-00674 of the Russian Foundation for Basic
%Research, and by the Ministry of Education and Science of the Russian
%Federation (projects 1.3843.2017 and 1.517.2016 respectively).}
\thanks{The work was supported by the Russian Foundation for Basic
Research (the joint project 17-51-150005-NCNI-a of RFBR and CNRS, France, and
the project 16-01-00674), and by the Ministry of Education and Science of the
Russian Federation (projects 1.3843.2017 and 1.517.2016).}



\begin{abstract}
Motivated by a problem in approximation theory, we find a necessary and
sufficient condition for a model (backward shift invariant) subspace
$K_\varTheta = H^2\ominus \varTheta H^2$ of the Hardy space $H^2$ to
contain a bounded univalent function.
\end{abstract}

\maketitle

\section{Introduction}

A famous theorem of Beurling says that any closed linear subspace of the
Hardy space $H^2$ in the unit disc $\mathbb D=\{z\colon |z|<1\}$ which is
invariant with respect to the shift $f(z)\mapsto zf(z)$ is of the form
$\varTheta H^2$ for some inner function $\varTheta$. The backward shift
invariant subspaces
$$
K_\varTheta = H^2\ominus \varTheta H^2
$$
(also known as model spaces)  play an exceptionally important role in modern
analysis. For their numerous applications in function and operator theory
(including functional models and spectral theory) we refer to \cite{Nik}.

Recently an interesting link was established between the model space theory
and approximation theory. This link is related with the concept of a
\emph{Nevanlinna domain}. Recall that a bounded simply connected domain
$\Omega\subset\mathbb C$ is said to be a Nevanlinna domain if there exist
two functions $u,v\in H^{\infty}(\Omega)$ such that the equality
$\overline{z}=u(z)/v(z)$ holds almost everywhere on $\partial\Omega$ in the
sense of conformal mappings (see \cite[def.~2.1]{cfp02}). This is equivalent
to the fact that some (and hence every) conformal mapping from $\mathbb D$
onto $\Omega$ admits a pseudocontinuation, and hence belongs to some model
space $K_\varTheta$ (see \cite{bf, fed06} where the concept of a Nevanlinna
domain is studied). It was shown by the third author in
\cite[theorem~1]{fed96} (see also \cite[theorem~2.2]{cfp02}) that for a
simple closed curve $\varGamma$, the bianalytic polynomials (that is the functions
of the form $p(z)+\overline{z}q(z)$, where $p$ and $q$ are polynomials in $z$)
are dense in $C(\Gamma)$ if and only if the domain $\Omega$
bounded by $\varGamma$ is not a Nevanlinna domain. This result contrasts with
the classical Mergelyan theorem and shows that new analytic obstacles appear
in the case of uniform approximation by polyanalytic polynomials. For more
general approximation results for polyanalytic polynomials involving the
notion of a Nevanlinna domain see \cite{bcf16, cfp02, fed96} and the survey
\cite{mpf13}.

Thus, the existence of univalent functions (e.g., with some
special properties) in model spaces turns out to be a noteworthy problem. In this paper we
describe those inner functions $\varTheta$ for which $K_\varTheta$ contains a
bounded univalent function. This question is trivial if $\varTheta(z_0)=0$
for some $z_0\in\mathbb D$ since in this case the univalent function
$f(z)=1\big/(1-\overline{z}_0z)$ belongs to $K_\varTheta$. Notice that all
known specific examples of Nevanlinna domains (see \cite{bf, fed06, maz15})
are obtained as images of $\mathbb D$ under mappings by special univalent
functions belonging to model spaces generated by appropriate Blaschke
products.

However, in the case when $\varTheta$ is a pure singular inner function the
problem becomes nontrivial. An essential difficulty here is that we know explicitly
only few elements of the space $K_\varTheta$. In particular, the reproducing
kernels of this space,
$$
k_\lambda(z)=
\dfrac{1-\overline{\varTheta(\lambda)}\varTheta(z)}{1-\overline{\lambda}z},
\quad \lambda\in\mathbb D,
$$
cannot be univalent since $\Theta$ itself does not belong to the Dirichlet space.

Recall that given a finite positive Borel measure $\mu$ on the unit circle
$\mathbb T=\{z\colon |z|=1\}$ which is singular with respect to Lebesgue
measure on $\mathbb T$, the corresponding singular inner function $S_\mu$ is
defined by
\begin{equation}
\label{sinn}
S_\mu(z)=
\exp\bigg(-\int_{\mathbb T}\frac{\zeta+z}{\zeta-z}\,d\mu(\zeta)\bigg),
\quad z\in\mathbb D.
\end{equation}

The univalence problem in $K_\varTheta$ was already addressed in
\cite[Section~3]{bf} where it was shown that if $K_{S_\mu}$ contains a
bounded univalent function, then there exists a Carleson set (a (closed) set of finite
entropy, see the definition below) $E\subset\mathbb T$ such that $\mu(E)>0$.
By the classical results of L.~Carleson, sets of finite entropy are precisely
those subsets of the unit circle that may serve as zero sets of smooth (up to the boundary)
analytic functions in the unit disc. H.~S.~Shapiro \cite{sh} showed that if
$\mu$ is supported by a Carleson set, then $K_{S_\mu}$ contains functions
from $C^\infty(\overline{\mathbb D})$. K.~Dyakonov and D.~Khavinson \cite{dk}
showed that, conversely, if $K_{S_\mu}$ contains a mildly smooth function
(e.g., from the standard Dirichlet space in $\mathbb D$), then $\mu(E)>0$ for
some Carleson set $E\subset\mathbb T$ (whence the necessity of this condition
for the existence of bounded univalent functions).

On the positive side, it was shown in \cite{bf} that if $\mu(E)>0$ for some
Carleson set $E$, then for a certain ``symmetrization'' of $S_\mu$, the
corresponding model space contains univalent functions. In particular, there
exist univalent functions in the space generated by the simplest ``atomic''
inner function $S(z)=\exp\Big(\dfrac{z+1}{z-1}\Big)$ or, equivalently, in
the Paley--Wiener space $PW_{[0,1]}$, the Fourier image of $L^2[0,1]$,
considered as a space of functions in the upper half-plane $\mathbb C_+$.
%Interestingly, we do not know any explicit examples of such univalent
%functions.

Notice that $K_{\varTheta_1}\subset K_{\varTheta_2}$ whenever $\varTheta_1$
divides $\varTheta_2$ (in the class of all inner functions). Thus, if $\mu$
has atoms, then $K_{S_\mu}$ contains bounded univalent functions.

The present paper completes the study of this problem by showing that the
condition ``{\it $\mu(E)>0$ for some  Carleson set $E$}'' is sufficient for the
existence of bounded univalent functions in $K_{S_\mu}$.

A set $E\subset\mathbb T$ is called a \emph{Carleson set} (a \emph{Beurling--Carleson set}) or a \emph{set of
finite entropy} if
$$
\int_{\mathbb T}\log\dist(\zeta,E)\,dm(\zeta)>-\infty,
$$
where $m$ is the normalized Lebesgue measure on $\mathbb T$. In this case $m(E)=0$. Furthermore, if
$E\subset\mathbb T$ is a closed set, $m(E)=0$, and
$\{I_\ell\}$ is the (at most countable) set of disjoint open arcs
$I_\ell\subset\mathbb T$ such that $\mathbb T\setminus E=\bigsqcup_\ell
I_\ell$, then $E$ is a Carleson set if and only if
$$
\Ent(E)=\sum_\ell |I_\ell|\log\frac{1}{|I_\ell|}<\infty,
$$
where $|I|$ stands for $m(I)$, and $\bigsqcup$ is the disjoint union here and in what follows.
We call the quantity
$\Ent(E)$ the \emph{entropy} of $E$.

Our main result is the following theorem:

\begin{theorem}\label{mainth}
Let $S$ be a singular inner function and let $\mu$ be the corresponding
\textup(positive singular\textup) measure on $\mathbb T$. The following
conditions are equivalent.

$(i)$ The space $K_S$ contains bounded univalent functions.

$(ii)$ There exists a Carleson set $E\subset\mathbb T$ such that $\mu(E)>0$.
\end{theorem}

An immediate corollary of Theorem~\ref{mainth} is

\begin{corollary}\label{hh}
A model space $K_\varTheta$ contains a bounded univalent function if and only
if either $\varTheta$ has a zero in $\mathbb D$ or $\varTheta$ is a singular
inner function such that the associated singular measure satisfies condition $(ii)$ of
Theorem~{\rm\ref{mainth}}.
\end{corollary}

We give two different proofs of Theorem~\ref{mainth}. The first one is based
on delicate estimates of entropy, which seem to be of independent interest.
The second proof is more straightforward.

%%%%%%%%%%%%%%%%%%%%%%%%%%%%%%%%%%%%%%%%%%%%%%%%%%%%%%%%%%%%%%%%%

\section{Preliminary observations}

Given a closed set $E\subset\mathbb T$ and an open arc $I$ we define the
\emph{local entropy} of $E$ with respect to $I$ by
$$
\Ent\nolimits_I(E)=\sum_\ell |I_\ell|\log\frac{1}{|I_\ell|},
$$
where $I_\ell$ are the open arcs such that $I\setminus E=\bigsqcup_{\ell}I_{\ell}$.

Note that for a set $E$ of zero Lebesgue measure (we will always consider only
such sets) we have $\sum_{\ell}|I_\ell|=|I|$ whence $\Ent_{I}(E)\ge|I|$ for
any arc $I$ with $|I|\le 1/e$. Also, there exists an absolute
constant $C>0$ such that
\begin{equation}\label{ess}
\int_I |\log\dist(\zeta,E)|\,dm(\zeta)\le C\Ent\nolimits_I(E),
\end{equation}
when $\sup_{\ell}|I_\ell|\le 1/e$ (the reverse inequality always holds with
constant $1$ when $E\cap I\not=\emptyset$).

In what follows, for $\gamma_1,\gamma_2\in\mathbb T$, we denote by
$[\gamma_1,\gamma_2]$ the arc of $\mathbb T$ with endpoints
$\gamma_1,\gamma_2$ in the positive (counter clockwise) direction.

The following lemma deals with the existence of smooth functions in $K_{S_\mu}$ with
uniform control on coefficients and plays the crucial role in our construction.

\begin{lemma}\label{key}
There exist absolute constants $\beta>0$, $\varepsilon\in(0,1/e)$ and $M\in\mathbb N$ such that
for every singular probability measure $\mu$ supported by a closed set $E\subset I$ for an arc $I$ with
$\Ent_I(E)\le\varepsilon$, there exists
a function $f\in K_{S_\mu}\cap C^3(\overline{\mathbb D})$ such that
$f(z)=\sum_{n\ge0}c_n z^n$,
\begin{equation}\label{proizv}
|c_1|\ge\beta
\end{equation}
and
\begin{equation}\label{coef}
\sum_{j=1}^{\infty}|c_{Mj+1}|(Mj+1)<\beta.
\end{equation}
\end{lemma}

\begin{proof}
Without loss of generality we assume that $I =[1,e^{i\eta}]$. %, $0<\eta<1/3$.
Note that $\eta<\varepsilon$ since $|I|\le\Ent_I(E)<\varepsilon$.

Let $S_0=\exp\Big(\dfrac{z+1}{z-1}\Big)$ be the atomic singular function corresponding to the
unit mass at the point $1$.
We begin by fixing a positive integer $k\ge10$ such that
\begin{equation}
\label{bedef}
\beta:=\frac12
\Big|\int_{\mathbb T}\overline{z}^2(1-\overline{z})^kS_0(z)\,dm(z)\Big|>0.
\end{equation}
Such choice of $k$ is possible; otherwise, $z^n(1-z)^{10}\perp S_0$, $n\ge 2$ and, hence,
$z^n\perp S_0$, $n\ge 2$, which is absurd.

Next, following the classical Carleson approach, let $F(z)$ be the outer
function such that $|F(z)|=\big(\dist(z,E)\big)^k$ a.e. on $\mathbb T$ (this
is possible since the function $z\mapsto \log\dist(z,E)$ is summable on the unit circle), i.e.,
$$
F(z)=\exp\Big(k\int_{\mathbb T}
\frac{\zeta+z}{\zeta-z}\log\dist(\zeta,E)\,dm(\zeta)\Big),
\quad z\in\mathbb D.
$$
The function $F$ is at least in the class $C^{[k/2]}(\overline{\mathbb
D})$. Indeed, it can be shown (see, e.g., \cite[Section 1]{carl}) that for $0\le j\le [k/2]$ we have
$$
|F^{(j)}(z)|\le C\big(\dist(z,E)\big)^{k-2j}, \quad z\in\overline{\mathbb D},
$$
where the value $C$ depends only on $k$ and $\Ent(E)$. More precisely,
if $\Ent(E)\le 1$, then
$C\le C_1(k)$. % where $C_1,C_2$ are two numeric constants.
By the hypothesis
of the lemma, for fixed absolute $k$
%, we can assume that $\Ent(E)$ is bounded by some absolute
%constant and so
this $C$ is an absolute constant.

Since $k\ge 10$ and $|S_\mu^{(j)}|\le C_2(\dist(z,E))^{-2j}$ for $0\le j\le 4$ and for some absolute constant
$C_2$, we conclude that $\overline{z}\overline{F}S_\mu\in C^4(\mathbb T)$. % if $k\ge10$.
Denote by $P_+$ the orthogonal projector from $L^2(\mathbb T)$ to
$H^2$. Then we have
$$
f:=P_+(\overline{z}\overline{F}S_\mu)\in H^2\cap C^3(\mathbb T)
$$
and $\|f\|_{C^3(\mathbb T)}\le B$ for some absolute constant $B$.
Set
$$
f(z)=\sum_{j\ge 0}c_jz^j.
$$
Then
$\sum_{j\ge1}|c_j|^2j^6\le B^2$, and, finally,
\begin{equation}\label{fast}
\sum_{j\ge1}j^2|c_j|\le B_1
\end{equation}
for another absolute constant $B_1$.

Now we show that in the conditions of Lemma, for sufficiently small
$\varepsilon>0$ we have
\begin{equation}
|c_1|=|f'(0)|=
\bigg|\int_{\mathbb T}\overline{z}^2\overline{F(z)}S_\mu(z)\,dm(z)\bigg|\ge\beta.
\label{newst}
\end{equation}
Then, choosing $M=B_1/\beta$ we deduce \eqref{coef} from
\eqref{fast}.

Let $\beta'=2^{-k-3}\beta$. First note that
\begin{equation}\label{rty}
\left\{
\begin{aligned}
\Big|\int_{[e^{-i\beta'},e^{i\beta'}]}
\overline{z}^2\overline{F(z)}S_\mu(z)\,dm(z)\Big|&\le\frac{\beta}{4}, \\
\Big|\int_{[e^{-i\beta'},e^{i\beta'}]}
\overline{z}^2(1-\overline{z})^kS_0(z)\,dm(z)\Big|&\le\frac{\beta}{4},
\end{aligned}\right.
\end{equation}
since $|F(z)|\le 2^k$, $|1-z|^k\le 2^k$ on $\mathbb T$ and the moduli of
other factors are bounded by $1$.

Next we show that for sufficiently small $\varepsilon$ and for $z\in \mathbb T\setminus [e^{-i\beta'},e^{i\beta'}]$ we have
\begin{equation}\label{out}
F(z)=(1-z)^{k}\big(1+O(\varepsilon^{1/3})\big),
\end{equation}
\begin{equation}
\label{inn}
S_\mu(z)=S_0(z)\big(1+O(\varepsilon)\big),
\end{equation}
where the constants involved in the $O$-estimates are determined by $\beta'$ and do
not depend on $z,I$ and $E$ provided that $\Ent_I(E)<\varepsilon$. Once these
estimates are established, \eqref{newst} follows immediately from \eqref{bedef},
\eqref{rty} and from the estimate
\begin{gather*}
\Big|\int_{\mathbb T\setminus[e^{-i\beta'},e^{i\beta'}]}
\overline{z}^2\overline{F(z)}S_\mu(z)\,dm(z)\Big|=\\
\Big|\int_{\mathbb T\setminus[e^{-i\beta'},e^{i\beta'}]}
\overline{z}^2(1-\overline{z})^kS_0(z)(1+O(\varepsilon^{1/3}))\,dm(z)\Big|\\
\ge 2\beta-\frac{\beta}{4}-O(\varepsilon^{1/3})>\frac{3\beta}{2},
\end{gather*}
%for $T_{\beta'}=\{z\in\mathbb T\colon |z-1|\ge\beta'\}$,
if $\varepsilon$ is sufficiently small.

\medskip\noindent
\textbf{Proof of \eqref{out}.} Put $\delta=\varepsilon^{1/3}$. We assume from
the very beginning that $\varepsilon$ is so small that $\delta<\beta'/10$. We
have
$$
\frac{F(z)}{(1-z)^k}=\exp\Big(k\int_{\mathbb T}\frac{\zeta+z}{\zeta-z}
\log\frac{\dist(\zeta,E)}{|1-\zeta|}\,dm(\zeta)\Big).
$$

Recall that $E\subset I=[1,e^{i\eta}]$ and $\eta<\varepsilon$. Then, for
$\zeta\in\mathbb T\setminus[e^{-i\delta},e^{i\delta}]$, we have
$$
\Bigl|\frac{\dist(\zeta,E)}{|1-\zeta|}-1\Bigr|
%\le\bigg|\frac{e^{i\eta}-\zeta}{1-\zeta}\bigg|=
%1+O\bigg(\bigg|\frac{1-e^{i\eta}}{1-\zeta}\bigg|\bigg)
=1+O\Big(\frac\eta\delta\Big).
$$
Hence,
$$
\Big|\int_{\mathbb T\setminus[e^{-i\delta},e^{i\delta}]\setminus [ze^{-i\delta},ze^{i\delta}]}
\frac{\zeta+z}{\zeta-z}\log\frac{\dist(\zeta,E)}{|1-\zeta|}\,dm(\zeta)\Big|=
O\Big(\frac\eta{\delta^2}\Big)=O(\varepsilon^{1/3}).
$$

For $\zeta\in [e^{-i\delta},e^{i\delta}]$, $z\in \mathbb T\setminus [e^{-i\beta'},e^{i\beta'}]$ we have $|\zeta-z|\ge\delta$ and, by a direct
estimate of the Schwarz kernel we get
\begin{multline*}
\Big|\int_{[e^{-i\delta},e^{i\delta}]}
\frac{\zeta+z}{\zeta-z}\log\frac{\dist(\zeta,E)}{|1-\zeta|}\,dm(\zeta)\Big|
\\ \le\frac{2}{\delta}\Big(\int_I|\log\dist(\zeta,E)|\,dm(\zeta)
+\int_I\log\frac{1}{|1-\zeta|}\,dm(\zeta)\\
+\int_{[e^{i\eta},e^{i\delta}]}\log\frac{|\zeta-1|}{|\zeta-e^{i\eta}|}
\,dm(\zeta)\Big)\\
\le\frac{2C}{\delta}\Ent\nolimits_I(E)+O\bigg(\eta\log\frac{1}{\eta}\bigg).
\end{multline*}
In the last inequality we use estimate \eqref{ess}. By the hypothesis
$\Ent_I(E)<\varepsilon$, we conclude that the whole integral is
$O(\varepsilon^{2/3})$.

Finally, to estimate the integral over the arc $J=[ze^{-i\delta},ze^{i\delta}]$, we use the following simple estimate: for any function
$\psi$ which is in $C^1$ on $J$ we have
\begin{multline*}
\Big|\int_J \frac{\zeta+z}{\zeta-z}\psi(\zeta)\,dm(\zeta)\Big|\\ \le
\int_J |\psi(\zeta)|\,dm(\zeta)
+2\Big|\int_J\frac{\psi(z)}{\zeta-z}\,dm(\zeta)\Big|
+2\Big|\int_J \frac{\psi(\zeta)-\psi(z)}{\zeta-z}
\,dm(\zeta)\Big|\\ \le
C\delta\big(\max_J |\psi|+|\psi(z)|
+\max_J |\psi'|\big)
\end{multline*}
for some absolute constant $C$.
We apply this estimate to
$$
\psi(\zeta)=\log\frac{\dist(\zeta,E)}{|1-\zeta|}.
$$
Since $|z-1|\ge\beta'$ and $\delta <\beta'/10$, we have $|\psi(\zeta)|\le
C\eta/\beta'$ and $|\psi'(\zeta)|\le C\eta/(\beta')^2$ when
$\zeta\in J$, for some absolute constant $C$. We conclude that the integral over $J$ is $O(\varepsilon^{4/3})$.
\medskip

\noindent
\textbf{Proof of \eqref{inn}.} The estimate for the inner factor is even more
straightforward. Using the fact that $\mu(I)=\mu(\mathbb T)=1$ we can write
\begin{gather*}
\frac{S_\mu(z)}{S_0(z)}=\exp\Big(\int_{\mathbb T}\Big(\frac{1+z}{1-z}-\frac{\zeta+z}{\zeta-z}\Big)
\,d\mu(\zeta)\Big)\\
=\exp\Big(\int_{I}\frac{2z(\zeta-1)}{(1-z)(\zeta-z)}\,d\mu(\zeta)\Big).
\end{gather*}
%We use that $\supp\mu\subset I$.
For every $\zeta\in I$ we have $|1-\zeta|\le\eta<\varepsilon$,
while $|\zeta-z|\ge\beta'/2$. Thus, $S_\mu(z)=S_0(z)\exp(O(\varepsilon))$.
\end{proof}

\begin{lemma}\label{spl}
Let $\mu$ be a non-trivial continuous singular measure supported by a closed set $E$ of finite
entropy. Then for any $\varepsilon,\delta>0$ there exists an arc $I$ such
that $0<\mu(I)<\delta$ and $\Ent_I(E)/\mu(I)<\varepsilon$.
\end{lemma}

\begin{proof}
Choose an open arc $I$ such that $0<\mu(I)<\delta$. Let $I\setminus E =
\bigcup_{j\ge 1} I_j$ with disjoint open arcs $I_j$. Choose $N$ such that
$$
\sum_{j=N+1}^\infty |I_j|\log\frac{1}{|I_j|}<\varepsilon \mu(I)
$$
and let $\overline{I}\setminus\bigcup_{j=1}^N {I_j}=\bigsqcup_{\ell=1}^{L}{J_\ell}$, where
$J_\ell$ are (closed) arcs. Assume that for any $1\le \ell\le L$ we have
$\Ent_{J_\ell}(E\cap J_\ell)\ge\varepsilon\mu(J_\ell)$. %(it is possible that
%$J_\ell$ is a point in which case $\Ent_{J_\ell}(E\cap J_\ell)$ is assumed to
%be $0$).
Then
$$
\sum_{j=N+1}^\infty |I_j|\log\frac{1}{|I_j|}=
\sum_{\ell=1}^L \Ent\nolimits_{J_\ell}(E\cap J_\ell)\ge
\varepsilon\sum_{\ell=1}^L \mu(J_\ell)=\varepsilon \mu(I),
$$
a contradiction. It remains to set $I=J_\ell$ for one of $J_\ell$ such that $\Ent_{J_\ell}(E\cap J_\ell)<\varepsilon\mu(J_\ell)$.
\end{proof}

Given $a\in\mathbb D$, consider the M\"obius transformation $\varphi_a:\mathbb D\to\mathbb D$,
$$
\varphi_a(z)=\frac{z-a}{1-\overline{a}z}.
$$

\begin{lemma}\label{mob}
Let $S=S_\mu$ be a singular inner function with $\supp(\mu)=E\subset I$,
where $I$ is an arc with endpoint $1$ and $|I|<1/100$. Let $r\in(9/10,1)$ be
such that $1-r>10|I|$. Put $\widetilde{S}=S\circ\varphi_{-r}$, the
composition of $S$ with the M\"obius transformation $\varphi_{-r}$. Then
\begin{enumerate}
\item[(i)] $\widetilde{S}$ is a singular inner function and the corresponding
singular measure $\widetilde\mu$ satisfies
\begin{equation}\label{ui}
\frac{\mu(\mathbb T)}{1-r}\le\widetilde\mu(\mathbb T)=\int_I\frac{1-r^2}{|\zeta-r|^2}\,d\mu(\zeta)\le
\frac{3\mu(\mathbb T)}{1-r}.
\end{equation}
\item[(ii)] There exists an arc $\widetilde{I}$ with endpoint $1$ such that
$\widetilde{E}:=\supp(\widetilde\mu)\subset\widetilde{I}$ and
$$
|\widetilde{I}|\le\frac{4|I|}{1-r},
\qquad
\Ent\nolimits_{\widetilde{I}}(\widetilde{E})\le
\frac{4}{1-r}\Ent\nolimits_{I}(E).
$$
\end{enumerate}
\end{lemma}

\begin{proof}
Clearly, $\widetilde{S}$ is an inner function which does not vanish in
$\mathbb D$. Therefore, $\widetilde{S}=S_{\widetilde\mu}$ for some singular measure $\widetilde\mu$.
We have $\exp(-\widetilde\mu(\mathbb
T))=|\widetilde{S}(0)|=|S(r)|$ and hence
$$
\widetilde\mu(\mathbb T)=\int_I\frac{1-r^2}{|\zeta-r|^2}\,d\mu(\zeta).
$$
Since $|I|<(1-r)/10$, we have $9(1-r)/10\le|\zeta-r|\le 11(1-r)/10$ for
$\zeta\in I$ and the estimate \eqref{ui} follows.

Since $\varphi_r$ is the inverse to $\varphi_{-r}$, we conclude that
$\widetilde\mu$ is supported by
$\widetilde{E}=\varphi_r(E)\subset\widetilde{I}=\varphi_r(I)$. Simple
estimates of $\varphi_r$ show that we have
$$
|J|\le|\varphi_r(J)|\le\frac{4}{1-r}|J|
$$
for any arc $J\subset I$. Hence, the local entropy also
increases at most by the factor $4(1-r)^{-1}$.
\end{proof}

\begin{lemma}\label{conf}
Let $\varTheta$ be an inner function and let $a\in\mathbb D$, $a\ne0$, be
such that $\varTheta(-a)\ne0$. Define
$\widetilde\varTheta=\varTheta\circ\varphi_a$. Let $f\in K_\varTheta$ and let
$g=f\circ\varphi_a$. Then there exists %a constant
$c_f\in\mathbb C$ such that
$g-c_f\in K_{\widetilde\varTheta}$.
\end{lemma}

\begin{proof}
In the proof we use the following criterion of being in
$K_\varTheta$ (see, e.g., \cite[Lecture~II]{Nikshift}): for a function $f\in
H^2$,
$$
f\in K_\varTheta\ \Longleftrightarrow\ \overline{z} \overline{f}\varTheta\in H^2,
$$
where the latter inclusion means that the function
$\overline{z}\overline{f}\varTheta$ defined on $\mathbb T$ coincides a.e.\ with some element of $H^2$.

Since $f\in K_\varTheta$ we have $\overline{z}\overline{f}\varTheta\in H^2$.
We take the composition with $\varphi_a$ on the right and denote
$$
h(z)=\overline{\Bigl(\frac{z-a}{1-\overline{a}z}\Bigr)}\overline{g(z)}
\widetilde\varTheta(z).
$$
Then $h\in H^2$.
Set $d_f=-ah(0)/\widetilde\varTheta(0)$ (note that
$\widetilde\varTheta(0)=\varTheta(-a)\ne0$). Clearly, $g-\overline{d_f}\in H^2$ and it
remains to show that $\overline{z}(\overline{g}-d_f)\widetilde\varTheta\in
H^2$. Indeed, for $z\in\mathbb T$,
$$
\overline{z}(\overline{g(z)}-d_f)\widetilde\varTheta(z)=
\frac1{z}\Bigl(h(z)\frac{z-a}{1-\overline{a}z}-d_f\widetilde\varTheta(z)\Bigr).
$$
By the choice of $d_f$ the function in brackets belongs to $H^2$ and vanishes at 0,
and hence the whole expression coincides with boundary values of some $H^2$-function.
It remains to set $c_f=\overline{d_f}$.
\end{proof}

%%%%%%%%%%%%%%%%%%%%%%%%%%%%%%%%%%%%%%%%%%%%%%%%%%%%%%%%%%%%%%%%%

\section{Proof of the main result}

Without loss of generality we assume that $\mu$ is a non-trivial continuous singular measure
supported by a closed set of finite entropy $E\subset I_0=[1, e^{i\alpha}]$, $\alpha\in(0,\pi/2]$, and
$1\in\supp(\mu)$. %!!! \beta<->\alpha

In what follows symbols $\mu_j$ denote different singular measures
supported by closed sets $E_j$. By $S_j$ we denote the singular inner functions
generated by $\mu_j$.

\medskip\noindent
\textbf{Step 1.} Fix the numbers $\beta,\varepsilon,M$ from Lemma~\ref{key}. By Lemma~\ref{spl} we can choose
an open arc $I$ with endpoint $1$ such that $\mu(I)\le1/4$ and
$$
|I|\le\Ent\nolimits_{I}(E\cap I)\le\varepsilon\frac{\mu(I)}{4M^2},
$$
%where $M$ and $\varepsilon$ are the numbers from Lemma~\ref{key}. Put
Set $\mu_1:=M^{-2}\mu|_I$ and denote by $S_1$ the corresponding singular inner
function. Note that $S_1^{M^2}$ is a divisor of our initial function $S_\mu$.
Later on, we will construct a univalent function inside $K_{S_1^{M^2}}$.

\medskip\noindent
\textbf{\bf Step 2.} We will now apply a conformal map to obtain from $\mu_1$
a probability measure whose entropy is much smaller than the mass (this
enables us to apply the key Lemma~\ref{key}).

By Lemma~\ref{mob}, we can choose $r\in(0,1)$ in such a way that
the singular
measure $\mu_2$ corresponding to the function $S_2=S_1\circ\varphi_{-r}$ has mass $1$ on
$\mathbb T$.
Then
$$
\frac{1-r}{\mu_1(\mathbb T)}=M^2\frac{1-r}{\mu(I)}\in [1, 3].
$$
Furthermore, by Lemma~\ref{mob} we have $\supp(\mu_2)=E_2\subset I_2=[1,e^{i\gamma}]$
for some $\gamma>0$, with
$$
\Ent\nolimits_{I_2}(E_2)\le \varepsilon.
$$

\medskip\noindent
\textbf{Step 3.} %Let $I_2=[1,e^{i\gamma}]$.
Let $\mu_3$ be the measure with support on the
arc $I_3=[1,e^{i\gamma/M}]$ and defined by
$$
\mu_3(A)=\mu_2(\{e^{iM\theta}\colon e^{i\theta}\in A\}), \qquad A\subset I_3.
$$
%i.e., the support of $\mu_3$ is the homothetical image of the support of the
%measure $\mu_2$.
Note that we still have
$$
\mu_3(\mathbb T)=1, \qquad\text{and}\quad \Ent\nolimits_{I_3}(E_3)\le \varepsilon,
$$
%where $I_3=[1,e^{i\gamma/M}]$ and $E_3=\supp(\mu_3)$.
For the corresponding
estimate of the local entropy note that $t\mapsto t\log\frac{1}{t}$ is an
increasing function on $(0,e^{-1})$.

\medskip\noindent
\textbf{Step 4.} We are now in a position to apply Lemma~\ref{key} to $\mu_3$
and the corresponding model space $K_{S_3}$: there exists a bounded function
$f(z)=\sum_{n\ge0}c_nz^n\in K_{S_3}$ such that
$$
f'(0)\ge\beta
$$
%(for the  absolute positive constant $\beta$ from Lemma~\ref{key})
and
\begin{equation}
\label{kk1}
\sum_{j=1}^\infty|c_{Mj+1}|(Mj+1)<\beta.
\end{equation}
%Note that by \eqref{kk1} the function $f$ is obviously bounded in $\mathbb
%D$.

Next we use the symmetrization trick whose application in a similar problem
was suggested by  M.~Putinar and H.~Shapiro \cite{ps} (it was subsequently
used in \cite{bf}). Take $\omega_M=e^{2\pi{}i/M}$ and consider the bounded analytic function
$$
\widetilde{f}(z)=
\frac{1}{M}\sum_{k=0}^{M-1}\overline\omega_M^k f(\omega_M^k z)=
\sum_{j=0}^{\infty}c_{Mj+1}z^{Mj+1}.
$$
Condition \eqref{kk1} guarantees that $\widetilde{f}$ is univalent in $\mathbb
D$ since $\mathop{\mathrm{Re}}\widetilde{f}'>0$ in $\mathbb D$. The function
$\widetilde{f}$ is no longer in $K_{S_3}$ but it belongs to $K_{S_4}$, where
$S_4$ is the singular inner function given by
$$
S_4(z)=\prod_{k=0}^{M-1}S_3(\omega_M^k z).
$$
It is associated with the measure $\mu_4$, which is the periodic expansion of
$\mu_3$ on the whole circle.

\medskip\noindent
\textbf{Step 5.} Now we apply a desymmetrization procedure as in \cite{bf}.
We have $\widetilde{f}(\omega_Mz)=\omega_Mf(z)$. Therefore, the function
$\check{f}(z)=\big(\widetilde{f}(z^{1/M})\big)^M$ is correctly defined in
$\mathbb D$ (does not depend on the choice of the branch of $z^{1/M}$) and is
also bounded and univalent in $\mathbb D$. A straightforward computation
shows that $\check{f}$ belongs to the space $K_{S_5}$ where $S_5(z) =
\big(S_4(z^{1/M})\big)^M$ (see \cite[Lemma~3]{bf}). Moreover, it is easy to
see that
$$
S_5(z)=
\exp\Big(-M^2\int_{\mathbb T}\frac{\zeta^M+z}{\zeta^M-z}\,d\mu_3(\zeta)\Big)=
\exp\Big(-M^2\int_{\mathbb T}\frac{\xi+z}{\xi-z}\,d\mu_2(\xi)\Big)
$$
and hence $S_5=S_2^{M^2}$. Thus, $\check{f}\in K_{S_2^{M^2}}$.

\medskip\noindent
\textbf{Step 6.} Our last step is the application of the conformal map
$\varphi_r$ which is inverse to $\varphi_{-r}$. Note that
$S_2^{M^2}\circ\varphi_r=S_1^{M^2}$. By Lemma~\ref{conf} there exists a
complex number $c$ such that $g=\check{f}\circ\varphi_r-c$ is in $K_{S_1^{M^2}}
\subset K_{S_\mu}$. It is clear that $g$ is bounded and univalent in $\mathbb
D$. \qed

\section{A short proof of Theorem~\ref{mainth}}

In this section we give another proof of Theorem~\ref{mainth} which is much
more straightforward than the proof given in the previous section.

We need to verify that for any singular inner function $S=S_\mu$ such that
$\supp(\mu)$ is a Carleson set on $\mathbb T$ there exists a univalent
function $f\in K_S$.

We start with taking an arbitrary non-trivial function $f_0\in K_S\cap
C^1(\overline{\mathbb D})$ which exists in view of \cite{dk}.

Notice that
%\begin{equation}\label{clark}
$$
f_0(z)=\frac{(1-S(z))}{2\pi i}\int_{\mathbb
T}\frac{\zeta+z}{\zeta-z}\widetilde{f}(\zeta)\,d\nu(\zeta),
$$
%\end{equation}
where $\nu$ is the corresponding Clark measure \cite{cla}, and
$\widetilde{f}$ is some function from $L^2(\nu)$. So $f_0$ has analytic
continuation to $\overline{\mathbb C}\setminus\supp(\mu)$. It is easy to see
that $f_0\notin H^{\infty}(\overline{\mathbb C}\setminus\overline{\mathbb
D})$. Indeed, otherwise boundary values of $f_0|_{\mathbb D}$ and
$f_0|_{\mathbb C\setminus\overline{\mathbb D}}$ coincide almost everywhere on
$\mathbb T$, which is clearly impossible whenever $f_0$ is non-constant.

So we can fix $a$, $1<|a|<2$, such that
$$
|f_0(a)|>100(\|f_0\|_{\infty,\mathbb T}+\|f_0'\|_{\infty,\mathbb T}).
$$
Put
$$
f(z)=\frac{1-Af_0(z)}{z-a}, \qquad A=\frac1{f_0(a)}.
$$
We have $f\in K_S$ and it remains to prove that $f$ is univalent in
$\overline{\mathbb D}$. Assume the contrary, i.e. $f(z)=f(w)$ for some $z\ne
w$, $z,w\in\overline{\mathbb D}$. Hence,
$$
1=aA\frac{f_0(w)-f_0(z)}{w-z}+A\frac{wf_0(z)-zf_0(w)}{w-z}=
aA\frac{f_0(w)-f_0(z)}{w-z}-Aw\frac{f_0(w)-f_0(z)}{w-z}+Af_0(w).
$$
It is easy to see that all three summands in right-hand side are bounded from
above by $1/10$. We arrive to a contradiction. $\qed$

\medskip

It is interesting to note that this proof of Theorem \ref{mainth} leads to an
explicit example of a univalent function in $PW_{[0,1]}$. It is easy to see
that function
$$f(z)=\frac{10(e^{iz}-1)-iz(e^{10}-1)}{z(z+10i)}$$
is univalent in $\mathbb{C}^+$ and $f\in PW_{[0,1]}$.

\section{Final remarks}

Among interesting problems concerning Nevanlinna domains one ought to
emphasize the question about possible irregularity of boundaries of
Nevanlinna domains. Several examples of Nevanlinna domains with sufficiently
irregular boundaries are known (see, for instance, \cite{bf, fed06, maz15}).
In particular, an example of a Jordan Nevanlinna domain with nonrectifiable
boundary was constructed in \cite{maz15}. All these examples are associated
with model spaces generated by Blaschke products and it seems interesting to
find similar examples in the case of singular inner functions.
%Notice in this connection that the function $f$ which was constructed at
%Step~4 of the proof of Theorem~1 belongs to the class $C^3(\overline{\mathbb
%D})$. Moreover, taking $k$ in Lemma~2.1 large enough, we can construct
%univalent function $f\in K_S\cap C^m(\overline{\mathbb D})$ for any integer
%$m\ge3$.

Finally, let us remark that some quantitative properties of univalent rational functions
(i.e., elements of $K_\Theta$ where $\Theta$ is a finite Blaschke product) were studied in \cite{toapp},
where estimates on the length of the boundary of $r(\mathbb D)$ are given in terms
of  the degree of the rational function $r$.

\subsubsection*{Acknowledgments} A part of this work was done in January--February 2017
during a ``research in pairs'' meeting at the CIRM (Luminy), France. Another
part of this work was done during the Simons Semester ``Emergent trends of
Complex Analysis and Functional Analysis'' hosted by IMPAN, Poland. The
authors thank the CIRM and IMPAN for their hospitality.

\vskip -0.75em

% ----------------------------------------------------------------
\end{document}